\documentclass{amsart}
\usepackage{amsfonts, amsbsy, amsmath, amssymb}

\newtheorem{thm}{Theorem}[section]
\newtheorem{lem}[thm]{Lemma}
\newtheorem{cor}[thm]{Corollary}

\newtheorem{conj}[thm]{Conjecture}

\newtheorem{rmk}[thm]{Remark}
\newtheorem{fac}[thm]{Fact}

\newtheorem{thm-con}[thm]{Theorem-Conjecture}
\numberwithin{equation}{section}

\theoremstyle{definition}

\allowdisplaybreaks

\newcommand{\x}{{\tt X}}
\newcommand{\f}{\Bbb F}

\begin{document}

\title{Permutation Polynomials of $\Bbb F_{q^2}$ of the form $a\x+\x^{r(q-1)+1}$}

\author[Xiang-dong Hou]{Xiang-dong Hou}
\address{Department of Mathematics and Statistics,
University of South Florida, Tampa, FL 33620}
\email{xhou@usf.edu}


\subjclass[2000]{}

\begin{abstract}
Let $q$ be a prime power, $2\le r\le q$, and $f=a\x+\x^{r(q-1)+1}\in\Bbb F_{q^2}[\x]$, where $a\ne 0$. The conditions on $r,q,a$ that are necessary and sufficient for $f$ to be a permutation polynomial (PP) of $\f_{q^2}$ are not known. (Such conditions are known under an additional assumption that $a^{q+1}=1$.) In this paper, we prove the following: (i) If $f$ is a PP of $\f_{q^2}$, then $\text{gcd}(r,q+1)>1$ and $(-a)^{(q+1)/\text{gcd}(r,q+1)}\ne 1$. (ii) For a fixed $r>2$ and  subject to the conditions that $q+1\equiv 0\pmod r$ and $a^{q+1}\ne 1$, there are only finitely many $(q,a)$ for which $f$ is a PP of $\f_{q^2}$. Combining (i) and (ii) confirms a recent conjecture regarding the type of permutation binomial considered here.
\end{abstract}

\maketitle

\section{Introduction}

Let $\f_q$ denote the finite field with $q$ elements. A polynomial $f\in\f_q[\x]$ is called a {\em permutation polynomial} (PP) of $\f_q$ if it induces a permutation of $\f_q$. In general, it is difficult to predict the permutation property of a polynomial from its algebraic appearance; even for binomials, the question remains challenging. Historical accounts and contemporary reviews of permutation polynomials can be found in a few book chapters and recent survey papers; see  \cite{Hou-Fq11-2015}, \cite{Hou-FFA-2015}, \cite[Ch. 7]{Lidl-Niederreiter-1997}, \cite[Ch. 8]{Mullen-Panario-2013}.

In this paper, we are concerned with the following question. Let $f=a\x+\x^{r(q-1)+1}\in\f_{q^2}[\x]$, where $2\le r\le q$ and $a\in\f_{q^2}^*$. When is $f$ a PP of $\f_{q^2}$? Our interest and curiosity in this question were inspired and elevated by several recent results. Under the assumption that $a^{q+1}=1$, the answer to question was provided by Zieve \cite{Zieve-arXiv1310.0076}.

\begin{thm}[\cite{Zieve-arXiv1310.0076}]\label{T1.1} 
Assume that $a^{q+1}=1$. Then $f$ is a PP of $\f_{q^2}$ if and only if $(-a)^{(q+1)/\text{\rm gcd}(r,q+1)}\ne 1$ and $\text{\rm gcd}(r-1,q+1)=1$.
\end{thm}

The proof of Theorem~\ref{T1.1} given in \cite{Zieve-arXiv1310.0076} uses degree one rational functions over $\f_{q^2}$ which permute the $(q+1)$-th roots of unity in $\f_{q^2}$. For that proof, the assumption that $a^{q+1}=1$ is essential. In the general situation, that is, without the assumption that $a^{q+1}=1$, the question
has been answered for $r=2,3,5,7$ in \cite{Hou-arXiv1404.1822, Hou-Lappano-JNT-2015, Lappano-ppt-2014}. It turns out that for $r=3,5,7$ and $a\in\f_{q^2}^*$ with $a^{q+1}\ne 1$, there are numerous but {\em finitely many} $(q,a)$ for which $f$ is a PP of $\f_{q^2}$. A conjecture has been formulated based on this observation:

\begin{conj}[\cite{Hou-FFA-2015, Lappano-ppt-2014}]\label{C1.2}
Let $r>2$ be a fixed prime. Under the assumption that $a^{q+1}\ne 1$ ($a\in\f_{q^2}^*$), there are only finitely many $(q,a)$ for which $f$ is a PP of $\f_{q^2}$.
\end{conj}

\begin{rmk}\label{R1.3}\rm 
Conjecture~\ref{C1.2} is false for $r=2$. It follows from \cite[Theorems A and B]{Hou-arXiv1404.1822} that $a\x+\x^{2(q-1)+1}$ ($a\in\f_{q^2}^*$) is a PP of $\f_{q^2}$ if and only if $q$ is odd and $(-a)^{(q+1)/2}=-1$ or $3$.
\end{rmk}

In the present paper, we confirm Conjecture~\ref{C1.2} by proving the following.

\begin{thm}\label{T1.4}
Assume that $f$ is a PP of $\f_{q^2}$. Then $\text{\rm gcd}(r,q+1)>1$ and $(-a)^{(q+1)/\text{\rm gcd}(r,q+1)}\ne 1$. In particular, if $r$ is a prime, then $q+1\equiv 0\pmod r$ and $(-a)^{(q+1)/r}\ne 1$.
\end{thm}

\begin{thm}\label{T1.5}
Let $r>2$ be fixed. Assume that $q\ge r$, $q+1\equiv 0\pmod r$ and $a^{q+1}\ne 1$ ($a\in\f_{q^2}^*$). Subject to these conditions, there are only finitely many $(q,a)$ for which $f$ is a PP of $\f_{q^2}$.
\end{thm}

It is well known that $f$ is a PP of $\f_{q^2}$ if and only $0$ is the unique root of $f$ in $\f_{q^2}$ and 
\[
\sum_{x\in\f_{q^2}}f(x)^s=0\qquad\text{for all}\ 1\le s\le q^2-2.
\]
The starting point of the paper is the computation of the power sums $\sum_{x\in\f_{q^2}}f(x)^s$, which is carried out in Section 2. Theorem~\ref{T1.4}, proved in Section 3, is a straightforward consequence of the computation in Section 2. The proof Theorem~\ref{T1.5} is rather involved; it is given in Section 4 preceded by a description of the proof strategy.

\section{The Power Sums}

Throughout the paper, we always assume that $f=a\x+\x^{r(q-1)+1}\in\f_{q^2}[\x]$, where $2\le r\le q$ and $a\ne 0$. All other conditions are considered additional. The following fact is obvious.

\begin{fac}\label{F2.1}\rm
$0$ is the only root of $f$ in $\f_{q^2}$ if and only if $(-a)^{(q+1)/\text{gcd}(r,q+1)}\ne 1$.
\end{fac}

For $0\le \alpha,\beta\le q-1$ with $(\alpha,\beta)\ne (0,0)$, we have
\[
\begin{split}
\sum_{x\in\f_{q^2}}f(x)^{\alpha+\beta q}\,&=\sum_{x\in\f_{q^2}^*}x^{\alpha+\beta q}(a+x^{r(q-1)})^{\alpha+\beta q}\cr
&=\sum_{x\in\f_{q^2}^*}x^{\alpha+\beta q}\sum_{i,j}\binom\alpha i\binom\beta j a^{\alpha+\beta q-i-jq}x^{r(q-1)(i+j q)}\cr
&=a^{\alpha+\beta q}\sum_{i,j}\binom\alpha i\binom\beta j a^{-i-j q}\sum_{x\in\f_{q^2}^*}x^{\alpha+\beta q+r(q-1)(i-j)}.
\end{split}
\]
In the above, the inner sum is $0$ unless $\alpha+\beta q\equiv 1\pmod{q-1}$, i.e., $0\le\alpha\le q-1$ and $\beta=q-1-\alpha$, in which case,
\[
\begin{split}
&\sum_{x\in\f_{q^2}}f(x)^{\alpha+(q-1-\alpha)q}\cr
=\,&a^{(\alpha+1)(1-q)}\sum_{i,j}\binom\alpha i\binom{q-1-\alpha}ja^{-i-jq}\sum_{x\in\f_{q^2}^*}x^{(q-1)[-\alpha-1+r(i-j)]}\cr
=\,&-a^{(\alpha+1)(1-q)}\sum_{-\alpha-1+r(i-j)\equiv 0\,(\text{mod}\,q+1)}\binom\alpha i\binom{q-1-\alpha}ja^{-i-jq}.
\end{split}
\]
As $i$ runs over the interval $[0,\alpha]$ and $j$ over the interval $[0,q-1-\alpha]$, the range of $-\alpha-1+r(i-j)$ is 
\begin{equation}\label{2.1}
I_\alpha:=\bigl[r\alpha-\alpha-1-r(q-1),\; r\alpha-\alpha-1\bigr].
\end{equation}
Thus
\begin{equation}\label{2.2}
\sum_{x\in\f_{q^2}}f(x)^{\alpha+(q-1-\alpha)q}=-a^{(\alpha+1)(1-q)}S_q(\alpha,a),
\end{equation}
where
\begin{equation}\label{2.3}
S_q(\alpha,a)=\sum_{-\alpha-1+r(i-j)\in I_\alpha\cap(q+1)\Bbb Z}\binom\alpha i\binom{q-1-\alpha}j a^{-i-jq}.
\end{equation}

\begin{lem}\label{L2.2}
Assume that $a^{q+1}=1$ but $(-a)^{(q+1)/d}\ne 1$, where $d=\text{\rm gcd}(r,q+1)$. Then, for $0\le\alpha\le q-1$, 
\begin{equation}\label{2.4}
S_q(\alpha,a)=
\begin{cases}
a^{-\alpha-1}&\text{if}\ (\alpha+1)(r-1)\equiv 0\pmod {q+1},\cr
0&\text{otherwise}.
\end{cases}
\end{equation}
\end{lem}

\begin{proof}
We only have to consider the case $\alpha+1\equiv 0\pmod d$. Let $k_0\in\Bbb Z$ be the smallest such that $\alpha+1+k_0(q+1)\in r\Bbb Z$ and 
\begin{equation}\label{2.5}
k_0(q+1)\ge r\alpha-\alpha-1-r(q-1).
\end{equation}
Then 
\begin{equation}\label{2.6}
\Bigl(k_0-\frac rd\Bigr)(q+1)<r\alpha-\alpha-1-r(q-1).
\end{equation}
We derive from \eqref{2.5} and \eqref{2.6} that 
\[
\begin{cases}
\displaystyle\Bigl(k_0+(d-2)\frac rd\Bigr)(q+1)< r\alpha-\alpha-1,\vspace{3mm}\cr
\displaystyle\Bigl(k_0+d\frac rd\Bigr)(q+1)> r\alpha-\alpha-1.
\end{cases}
\]
Therefore
\begin{equation}\label{2.7}
\bigl[\alpha+1+\bigl(I_\alpha\cap(q+1)\Bbb Z\bigr)\bigr]\cap r\Bbb Z=\alpha+1+k_0(q+1)+\frac{r(q+1)}dL,
\end{equation}
where $L=\{0,\dots,d-1\}$ or $\{0,\dots,d-2\}$.

If $L=\{0,\dots,d-2\}$, then
\[
\Bigl(k_0+(d-1)\frac rd\Bigr)(q+1)>r-\alpha-1,
\]
which is equivalent to 
\begin{equation}\label{2.8}
\alpha+1+k_0(q+1)+\frac{d-1}d(q+1)r\ge r\alpha+r.
\end{equation}
Also note that \eqref{2.6} is equivalent to 
\begin{equation}\label{2.9}
\alpha+1+k_0(q+1)-\frac 1d(q+1)r\le r\alpha-rq.
\end{equation}
Taking the difference of \eqref{2.8} and \eqref{2.9}, we see that the equal sign holds in both \eqref{2.8} and \eqref{2.9}. Hence
\begin{equation}\label{2.10}
\alpha+1+k_0(q+1)=\Bigl(\alpha+1-\frac{d-1}d(q+1)\Bigr)r,
\end{equation}
which further implies that $(\alpha+1)(r-1)\equiv 0\pmod {q+1}$. On the other hand, if $(\alpha+1)(r-1)\equiv 0\pmod{q+1}$ and $k_0$ is given by \eqref{2.10}, reversing the above argument gives $L=\{0,\dots,d-2\}$. Thus we have proved that $L=\{0,\dots,d-2\}$ if and only if $(\alpha+1)(r-1)\equiv 0\pmod{q+1}$.

Now we have 
\[
\begin{split}
S_q(\alpha,a)\,&=\sum_{-\alpha-1+r(i-j)\in I_\alpha\cap(q+1)\Bbb Z}\binom\alpha i\binom{q-1-\alpha}j a^{-i+j}\kern 1.1cm\text{(since $a^{q+1}=1$)}\cr
&=\sum_{i-j\in\frac 1r(\alpha+1+k_0(q+1))+\frac{q+1}dL}\binom\alpha i\binom{q-1-\alpha}j a^{-i+j}\kern 0.75cm\text{(by \eqref{2.7})}\cr
&=\sum_{\alpha-i+j\in \alpha-\frac 1r(\alpha+1+k_0(q+1))-\frac{q+1}dL}\binom\alpha{\alpha-i}\binom{q-1-\alpha}ja^{-i+j}\cr
&=\sum_{l\in L}a^{-\frac 1r(\alpha+1+k_0(q+1))-\frac{q+1}dl}\binom{q-1}{\alpha-\frac 1r(\alpha+1+k_0(q+1))-\frac{q+1}dl}\cr
&=\sum_{l\in L}a^{-\frac 1r(\alpha+1+k_0(q+1))-\frac{q+1}dl}(-1)^{\alpha-\frac 1r(\alpha+1+k_0(q+1))-\frac{q+1}dl}\cr
&=(-1)^\alpha(-a)^{-\frac 1r(\alpha+1+k_0(q+1))}\sum_{l\in L}(-a)^{-\frac{q+1}dl}.
\end{split}
\]
(For the next-to-last step, note that $\alpha-\frac 1r(\alpha+1+k_0(q+1))-\frac{q+1}dl\in \alpha-\frac 1r(\alpha+1+I_\alpha)=[0,q-1]$ by \eqref{2.7} and \eqref{2.1}.) If $(\alpha+1)(r-1)\not\equiv 0\pmod{q+1}$, then $L=\{0,\dots,d-1\}$ and $\sum_{l\in L}(-a)^{-\frac{q+1}dl}=0$. If $(\alpha+1)(r-1)\equiv 0\pmod {q+1}$, then $L=\{0,\dots,d-2\}$ and hence 
\[
\begin{split}
S_q(\alpha,a)\,&=(-1)^{\alpha+1}(-a)^{-\frac 1r(\alpha+1+k_0(q+1))}(-a)^{-\frac{q+1}d(d-1)}\cr
&=(-1)^{\alpha+1}(-a)^{-(\alpha+1-\frac{d-1}d(q+1))-\frac{q+1}d(d-1)}\kern1.2cm \text{(by \eqref{2.10})}\cr
&=a^{-\alpha-1}.
\end{split}
\]
\end{proof}

\begin{rmk}\label{R2.3}\rm
Note that $\text{gcd}(r-1,q+1)=1$ if and only if there is no $0\le \alpha\le q-1$ such that $(\alpha+1)(r-1)\equiv 0\pmod{q+1}$. Hence Lemma~\ref{L2.2} and Fact~\ref{F2.1} provide an alternate proof of Theorem~\ref{T1.1}.
\end{rmk}

\begin{lem}\label{L2.4}
Assume that $q+1\equiv 0\pmod r$, $0\le \alpha\le q-1$, $\alpha+1\equiv 0\pmod r$ and $q\ge(r-1)(\alpha+2)$. Then
\begin{equation}\label{2.11}
S_q(\alpha,a)=(-a)^{\frac{\alpha+1}rq}\Bigl(\sum_{l=0}^{r-1}v^l\Bigr)g_\alpha(v),
\end{equation}
where $v=(-a)^{-\frac{q+1}rq}$ and
\begin{equation}\label{2.12}
g_\alpha(\x)=(\x-1)\sum_{l=0}^{r-1}\x^l\sum_{i=1}^\alpha(-1)^i\binom\alpha i\binom{i+\alpha-\frac{\alpha+1}r+\frac lr}\alpha\frac{\x^{ri}-1}{\x^r-1}+(-1)^\alpha.
\end{equation}
\end{lem}

\begin{proof}
Since $q\ge(r-1)(\alpha+2)$, it is easy to see that $I_\alpha\cap(q+1)\Bbb Z=\{-(r-1)(q+1),\dots,-(q+1),0\}$. Therefore,
\begin{equation}\label{2.13}
\begin{split}
&S_\alpha(\alpha,a)\cr
&=\sum_{-\alpha-1+r(i-j)\in I_\alpha\cap(q+1)\Bbb Z}\binom\alpha i\binom{q-1-\alpha}ja^{-i-jq}\cr
&=\sum_{l=0}^{r-1}\sum_{-\alpha-1+r(i-j)=-l(q+1)}\binom\alpha i\binom{q-1-\alpha}ja^{-i-jq}\cr
&=\sum_{l=0}^{r-1}\sum_{i=0}^\alpha\binom\alpha i\binom{q-1-\alpha}{\frac 1r(l(q+1)-\alpha-1)+i}a^{-i-[\frac 1r(l(q+1)-\alpha-1)+i]q}\cr
&=a^{\frac{\alpha+1}rq}\sum_{l=0}^{r-1}\sum_{i=0}^\alpha\binom\alpha i\binom{-1-\alpha}{\frac 1r(l(q+1)-\alpha-1)+i}a^{-\frac{q+1}rq(l+ri)}\cr
&\kern 5cm (\text{since}\ \frac 1r(l(q+1)-\alpha-1)+i\le q-1)\cr
&=a^{\frac{\alpha+1}rq}\sum_{l=0}^{r-1}\sum_{i=0}^\alpha\binom\alpha i(-1)^{\frac 1r(l(q+1)-\alpha-1)+i}\binom{\alpha+\frac 1r(l(q+1)-\alpha-1)+i}\alpha a^{-\frac{q+1}rq(l+ri)}\cr
&=(-a)^{\frac{\alpha+1}rq}\sum_{l=0}^{r-1}\sum_{i=0}^\alpha(-1)^i\binom\alpha i\binom{i+\alpha-\frac{\alpha+1}r+\frac lr}\alpha v^{l+ri}.
\end{split}
\end{equation}
We also have
\begin{equation}\label{2.14}
\begin{split}
&\sum_{l=0}^{r-1}\sum_{i=0}^\alpha(-1)^i\binom\alpha i\binom{i+\alpha-\frac{\alpha+1}r+\frac lr}\alpha v^l\cr
=\,&\sum_{l=0}^{r-1}v^l\sum_{i=0}^\alpha\binom\alpha{\alpha-i}(-1)^{\frac 1r(l(q+1)-\alpha-1)}\binom{-1-\alpha}{\frac 1r(l(q+1)-\alpha-1)+i}\cr
=\,&\sum_{l=0}^{r-1}v^l(-1)^{\frac 1r(l(q+1)-\alpha-1)}\binom{-1}{\alpha+\frac 1r(l(q+1)-\alpha-1)}\cr
=\,&(-1)^\alpha\sum_{l=0}^{r-1}v^l.
\end{split}
\end{equation}
Combining \eqref{2.13} and \eqref{2.14} gives
\begin{align*}
&S_q(\alpha,a)\cr
&=(-a)^{\frac{\alpha+1}rq}\sum_{l=0}^{r-1}\sum_{i=0}^\alpha(-1)^i\binom\alpha i\binom{i+\alpha-\frac{\alpha+1}r+\frac lr}\alpha (v^{l+ri}-v^l+v^l)\cr
&=(-a)^{\frac{\alpha+1}rq}\Bigl[\sum_{l=0}^{r-1}v^l\sum_{i=0}^\alpha(-1)^i\binom\alpha i\binom{i+\alpha-\frac{\alpha+1}r+\frac lr}\alpha(v^{ri}-1)+(-1)^\alpha\sum_{l=0}^{r-1}v^l\Bigr]\cr
&=(-a)^{\frac{\alpha+1}rq}\Bigl(\sum_{l=0}^{r-1}v^l\Bigr)\Bigl[(v-1)\sum_{l=0}^{r-1}v^l\sum_{i=0}^\alpha(-1)^i\binom\alpha i\binom{i+\alpha-\frac{\alpha+1}r+\frac lr}\alpha\frac{v^{ri}-1}{v^r-1}+(-1)^\alpha\Bigr]\cr
&=(-a)^{\frac{\alpha+1}rq}\Bigl(\sum_{l=0}^{r-1}v^l\Bigr)g_\alpha(v).
\end{align*}
\end{proof}

In \eqref{2.12}, we may treat $g_\alpha(\x)$ as a polynomial over $\Bbb Z[1/r]$. The reason is that for all $i\in\Bbb Z$ and $j\in\Bbb N$, $\binom{i/r}j$ is a $p$-adic integer for all primes $p\nmid r$, and hence $\binom{i/r}j\in\Bbb Z[1/r]$. We have
\[
\begin{split}
g_\alpha(0)\,&=-\sum_{i=1}^\alpha(-1)^i\binom\alpha i\binom{i+\alpha-\frac{\alpha+1}r}\alpha+(-1)^\alpha\cr
&=-\sum_{i=1}^\alpha\binom\alpha i\binom{-\alpha-1}{i-\frac{\alpha+1}r}(-1)^{\frac{\alpha+1}r}+(-1)^\alpha\cr
&=-(-1)^{\frac{\alpha+1}r}\sum_{i=0}^\alpha\binom\alpha{\alpha-i}\binom{-\alpha-1}{i-\frac{\alpha+1}r}+(-1)^\alpha\cr
&=-(-1)^{\frac{\alpha+1}r}\binom{-1}{\alpha-\frac{\alpha+1}r}+(-1)^\alpha=-(-1)^{\frac{\alpha+1}r}(-1)^{\alpha-\frac{\alpha+1}r}+(-1)^\alpha=0.
\end{split}
\]
Hence
\begin{equation}\label{2.15}
h_\alpha(\x):=\x^{-1}g_\alpha(\x)\in(\Bbb Z[1/r])[\x].
\end{equation}

\section{Proof of Theorem~\ref{T1.4}}

Given that $f$ is a PP of $\f_{q^2}$, we show that $\text{gcd}(r,q+1)>1$. Assume to the contrary that $\text{gcd}(r,q+1)=1$. Then there exists a unique $j_0\in\{0,1,\dots,q\}$ such that $-1-rj_0\equiv 0\pmod{q+1}$. Note that $j_0\ne q$ since otherwise $r\equiv 1\pmod{q+1}$, which would imply that $r\ge q+2$, a contradiction. By \eqref{2.3}, 
\[
S_q(0,a)=\sum_{-1-rj\equiv 0\, (\text{mod}\,q+1)}\binom{q-1}ja^{-jq}=\binom{q-1}{j_0}a^{-j_0q}=(-1)^{j_0}a^{-j_0q}\ne 0,
\]
which is a contradiction.

\section{Proof of Theorem~\ref{T1.5}}

Fix $r\ge 3$ and assume that $f=a\x+\x^{r(q-1)+1}$ is a PP of $\f_{q^2}$, where $q\ge r$, $q+1\equiv 0\pmod r$ and $a^{q+1}\ne 1$ ($a\in\f_{q^2}^*$). Our goal is to show that there only finitely many possibilities for $q$. The strategy and the outline of the proof are the following.
\begin{itemize}
  \item [1.] We show that $\text{gcd}_{\Bbb Q[\x]}\{h_\alpha:\alpha>0,\ \alpha\equiv -1\pmod r\}=1$, where $h_\alpha$ is defined in \eqref{2.15}. It follows that there are only finitely many possibilities for $p:=\text{char}\,\f_q$.
  \item [2.] If for a certain characteristic $p$, there are infinitely many possibilities for $q$, then we must have $p\in\{2,3,5\}$.
  \item [3.] For each $p\in\{2,3,5\}$, we prove that there are only finitely many possibilities for $q$.
\end{itemize} 


\subsection{Finiteness of the possibilities of the characteristic}\

Let $\tau>1$ be a prime power with $\tau\equiv -1\pmod r$ and $\text{char}\,\f_\tau=t$. Set $k=(\tau+1)/r$. Note that the polynomial $g_\alpha(\x)$ in \eqref{2.12} can be treated as a polynomial in $\f_t[\x]$. For Lemmas~\ref{L4.1} and \ref{L4.3}, we only require that $r\ge 2$. Lemma~\ref{L4.2} requires that $r\ge m$ and Lemmas~\ref{L4.4} -- \ref{L4.6} require that $r\ge 3$.
  
\begin{lem}\label{L4.1} Assume that $x\ne 0$ is a common root of $g_\tau$, $g_{\tau^3}$ and $g_{\tau^4-2}$ in some extension of $\f_t$. Then, $x\in\f_{\tau^2}^*$, $x^r\in\f_\tau^*\setminus\{1\}$, and 
\begin{equation}\label{4.1}
x^r=\frac{x(rx+1-r)}{(r+1)x-r},
\end{equation}
where $(r+1)x-r\ne 0$.
\end{lem}

\begin{proof}
All polynomials considered in this proof are in characteristic $t$.

$1^\circ$ Let $\alpha=\tau^{2m+1}$, where $m\ge 0$ is an integer. Note that $\alpha\equiv -1\pmod r$ and that \eqref{2.12} gives
\begin{equation}\label{4.2}
-g_\alpha=(\x-1)\frac{x^{\alpha r}-1}{\x^r-1}\sum_{l=0}^{r-1}\binom{2\alpha-\frac{\alpha+1}r+\frac lr}\alpha x^l+1.
\end{equation}
Let $k'=(\alpha+1)/r$. In the ring $\Bbb Z_t$ of $t$-adic integers, 
\[
\frac 1r=\frac{k'}{1+\alpha}\equiv k'(1-\alpha)\pmod{\alpha^2}.
\]
Note that $k'\equiv\frac 1r\equiv k\pmod t$. For $l\in\Bbb Z$,
\[
\begin{split}
2\alpha-\frac{\alpha+1}r+\frac lr\,&\equiv 2\alpha-k'+lk'(1-\alpha)\pmod{\alpha^2}\cr
&=k'(l-1)+\alpha(2-k'l).
\end{split}
\]
When $1\le l\le r-1$, we have
\[
\begin{split}
\binom{2\alpha-\frac{\alpha+1}r+\frac lr}\alpha\,&=\binom{k'(l-1)+\alpha(2-k'l)}\alpha\cr
&=\binom{2-k'l}1\kern 3cm \text{(since $0\le k'(l-1)<\alpha$)}\cr
&=2-k'l=2-kl.
\end{split}
\]
When $l=0$,
\[
\binom{2\alpha-\frac{\alpha+1}r}\alpha=\binom{2\alpha-k'}\alpha=\binom{\alpha-k'+\alpha}\alpha=1.
\]
Therefore,
\begin{equation}\label{4.3}
\begin{split}
&-g_\alpha\cr
&=(\x-1)\frac{\x^{\alpha r}-1}{\x^r-1}\Bigl[1+\sum_{l=1}^{r-1}(2-kl)\x^l\Bigr]+1\cr
&=(\x^r-1)^{\alpha-1}(\x-1)\Bigl[-1+\sum_{l=0}^{r-1}(2-kl)\x^l\Bigr]+1\cr
&=(\x^r-1)^{\alpha-1}(\x-1)\Bigl[-1+2\frac{\x^r-1}{\x-1}-k\x\bigl(r\x^{r-1}(\x-1)^{-1}-(\x^r-1)(\x-1)^{-2}\bigr)\Bigr]+1\cr
&=(\x^r-1)^{\alpha-1}\Bigl[\Bigl(1+\frac{k\x}{\x-1}\Bigr)(\x^r-1)-\x\Bigr]+1.
\end{split}
\end{equation}
(In the last step, we used the relation $kr\equiv 1\pmod t$.) Since $g_\tau(x)=0=g_{\tau^3}(x)$, it follows from \eqref{4.3} that 
\[
(x^r-1)^{\tau-1}=(x^r-1)^{\tau^3-1}.
\]
Therefore $(x^r-1)^\tau=(x^r-1)^{\tau^3}$, which gives $x^{r\tau^2}=x^r$; that is, $x^r\in\f_{\tau^2}$. Clearly, $x^r\ne 1$. (Otherwise, by \eqref{4.3}, we have $-g_\tau(x)=1\ne 0$, which is a contradiction.) Using \eqref{4.3}, the condition $g_\tau(x)=0$ becomes
\begin{equation}\label{4.4}
-(x^r-1)^{\tau-1}x(x-1)+(x^r-1)^\tau\bigl[(1+k)x-1\bigr]+x-1=0.
\end{equation}
Treating \eqref{4.4} as a quadratic equation in $x$ with coefficients in $\f_{\tau^2}$, we conclude that $x\in\f_{\tau^4}$.

\medskip
$2^\circ$ Let $\beta=\tau^4-2$. (Note that $\beta\equiv -1\pmod r$.) Let $k''=(\beta+1)/r$. In $\Bbb Z_t$,
\[
\frac 1r=\frac{k''}{\tau^4-1}\equiv -k''\pmod{\tau^4},
\]
and $k''\equiv-\frac 1r\equiv -k\pmod t$. For $1\le i\le \alpha$ and $0\le l\le r-1$, we have 
\[
i+\beta-\frac{\beta+1}r+\frac lr\equiv i-2-k''-lk''\pmod{\tau^4}.
\]
Hence
\begin{equation}\label{4.5}
\begin{split}
&\binom{i+\beta-\frac{\beta+1}r+\frac lr}\beta=\binom{i-2-k''(1+l)}\beta\cr
=\,&
\begin{cases}
-1&\text{if}\ i-2-k''(1+l)\equiv -1\pmod{\tau^4},\cr
1&\text{if}\ i-2-k''(1+l)\equiv -2\pmod{\tau^4},\cr
0&\text{otherwise}.
\end{cases}
\end{split}
\end{equation}
(In the last step, we used the fact that the base $t$ digits of $\beta$ are $(t-2,t-1,\dots,t-1)$.) Since
\[
i-2-k''(1+l)\le \tau^4-4-k''<\tau^4-4
\]
and 
\[
i-2-k''(1+l)\ge -1-k''r=-1-(\beta+1)=-\tau^4,
\]
we see that for $c=-1$ and $-2$, $i-2-k''(1+l)\equiv c\pmod{\tau^4}$ if and only if $i-2-k''(1+l)=c$. Therefore we can rewrite \eqref{4.5} as
\begin{equation}\label{4.6}
\binom{i+\beta-\frac{\beta+1}r+\frac lr}\beta=
\begin{cases}
-1&\text{if}\ i-2-k''(1+l)= -1,\cr
1&\text{if}\ i-2-k''(1+l)=-2,\cr
0&\text{otherwise}.
\end{cases}
\end{equation}
By \eqref{2.12} and \eqref{4.6}, we have
\begin{equation}\label{4.7}
\begin{split}
0\,&=-g_\beta(x)\cr
&=-(x-1)\sum_{l=0}^{r-1}x^l\Bigl[(-1)^{k''(1+l)}\binom\beta{k''(1+l)+1}\frac{x^{r(k''(1+l)+1)}-1}{x^r-1}\cr
&\kern 3cm +(-1)^{k''(1+l)}\binom\beta{k''(1+l)}\frac{x^{rk''(1+l)}-1}{x^r-1}\Bigr]+1\cr
&=-(x-1)\sum_{l=0}^{r-1}x^l(-1)^{k''(1+l)}\binom{\tau^4-2}{k''(1+l)+1}+1.
\end{split}
\end{equation}
(For the last step, note that $x^{rk''}=x^{\beta+1}=x^{\tau^4-1}=1$ since $x\in\f_{\tau^4}^*$.)
Note that 
\[
k''(1+l)+1
\begin{cases}
<\tau^4&\text{if}\ 0\le l<r-1,\cr
=\tau^4&\text{if}\ l=r-1.
\end{cases}
\]
Hence
\[
\begin{split}
&\binom{\tau^4-2}{k''(1+l)+1}\cr
=\,&\begin{cases}
\displaystyle\binom{-2}{k''(1+l)+1}=(-1)^{k''(1+l)+1}\bigl(k''(1+l)+2\bigr)&\text{if}\ 0\le l<r-1,\vspace{2mm}\cr
0&\text{if}\ l=r-1.
\end{cases}
\end{split}
\]
Thus \eqref{4.7} becomes
\begin{equation}\label{4.8}
\begin{split}
0\,&=(x-1)\sum_{l=0}^{r-2}\bigl(k''(1+l)+2\bigr)x^l+1\cr
&=(x-1)\Bigl[-x^{r-1}+\sum_{l=0}^{r-1}\bigl(2-k(1+l)\bigr)x^l\Bigr]+1\kern 1cm \text{(since $k''\equiv -k\pmod t$)}\cr
&=(x-1)\Bigl[-x^{r-1}+(2-k)\frac{x^r-1}{x-1}-kx\bigl(rx^{r-1}(x-1)^{-1}-(x^r-1)(x-1)^{-2}\bigr)\Bigr]+1\cr
&=-x^r+x^{r-1}+(2-k)(x^r-1)-x^r+\frac{kx(x^r-1)}{x-1}+1\cr
&=\frac k{x-1}(x^r-1)+x^{r-1}-1.
\end{split}
\end{equation}
Multiplying \eqref{4.8} by $x$ gives 
\begin{equation}\label{4.9}
0=\Bigl(1+\frac{kx}{x-1}\Bigr)(x^r-1)-x+1.
\end{equation}
On the other hand, by \eqref{4.3}, 
\begin{equation}\label{4.10}
0=-g_\tau(x)=(x^r-1)^{\tau-1}\Bigl[\Bigl(1+\frac{kx}{x-1}\Bigr)(x^r-1)-x\Bigr]+1.
\end{equation}
Combining \eqref{4.9} and \eqref{4.10} gives $(x^r-1)^{\tau-1}=1$, which implies that $x^r\in\f_\tau$. It follows from \eqref{4.4} that $x\in\f_{\tau^2}$. 

Finally, by \eqref{4.9}, we have
\[
x^r\bigl[(r+1)x-r\bigr]=x(rx+1-r).
\] 
If $(r+1)x-r=0$, then $rx+1-r=0$; the two equations imply that $x=1$, which is a contradiction. Hence we have $(r+1)x-r\ne 0$ and 
\eqref{4.1} follows. 
\end{proof}

\begin{lem}\label{L4.2}
Let $x\in\overline\f_t$ (the algebraic closure of $\f_t$) be such that $x^r\in\f_\tau^*\setminus\{1\}$. Let $m$ be a positive integer, and assume that $r\ge m$ and $k:=(\tau+1)/r>2(m-1)$. Then
\begin{equation}\label{4.11}
\begin{split}
g_{m-1+m\tau}(x)
=\,&(x-1)\sum_{\substack{0\le i_0\le m-1\cr 0\le i_1\le m}}(-1)^{i_0+i_1}\binom{m-1}{i_0}\binom m{i_1}\frac{x^{r(i_0+i_1)}-1}{x^r-1}\cr
&\cdot\Bigl[\sum_{l=0}^{r-1}\binom{i_0+m-1+(l-m)k}{m-1}\binom{i_1+m-lk}mx^l\cr
&-\sum_{l=0}^{m-1}\binom{i_0+m-1+(l-m)k}{m-1}\binom{i_1+m-lk}{m-1}x^l\Bigr]-1.
\end{split}
\end{equation}
\end{lem}

\begin{proof}
Let $\alpha=m-1+m\tau$. For any integer $i\ge 0$, we have $\binom\alpha i=0$ unless $i=i_0+i_1\tau$, where $0\le i_0\le m-1$ and $0\le i_1\le m$, in which case,
\[
\binom\alpha i=\binom{m-1}{i_0}\binom m{i_1}.
\]
Thus by \eqref{2.12},
\begin{equation}\label{4.12}
\begin{split}
g_\alpha(x)
=\,&(x-1)\sum_{l=0}^{r-1}x^l\sum_{\substack{0\le i_0\le m-1\cr 0\le i_1\le m}}(-1)^{i_0+i_1}\binom{m-1}{i_0}\binom m{i_1}\cr
&\cdot\binom{i_0+i_1\tau+\alpha-\frac{\alpha+1}r+\frac lr}\alpha\frac{x^{r(i_0+i_1)}-1}{x^r-1}-1.
\end{split}
\end{equation}
We have 
\begin{equation}\label{4.13}
\begin{split}
&i_0+i_1\tau+\alpha-\frac{\alpha+1}r+\frac lr\cr
\equiv\,&i_0+i_1\tau+m-1+m\tau-mk+lk(1-\tau)\pmod{\tau^2}\cr
=\,&\bigl(i_0+m-1+(l-m)k\bigr)+(i_1+m-lk)\tau.
\end{split}
\end{equation}
When $0\le l\le m-1$,
\begin{equation}\label{4.14}
\begin{cases}
i_0+m-1+(l-m)k+\tau\le 2m-2-k+\tau<\tau,\vspace{2mm}\cr
i_0+m-1+(l-m)k+\tau\ge m-1-mk+\tau=\tau+m-1-\displaystyle\frac mr(\tau+1)\ge 0.
\end{cases}
\end{equation}
When $m\le l\le r-1$,
\begin{equation}\label{4.15}
0\le i_0+m-1+(l-m)k\le 2m-2+(r-1-m)k=2m-2-(m+1)k+\tau+1<\tau.
\end{equation}
(To see the last inequality, consider the cases $m=1$ and $m>1$ separately; note that when $m>1$, $k>2(m-1)\ge 2$.) Therefore,
\begin{equation}\label{4.16}
\begin{split}
&\binom{i_0+i_1\tau+\alpha-\frac{\alpha+1}r+\frac lr}\alpha \vrule height 0pt width 0pt depth 4mm\cr
=\,&\left(\!\!
\begin{array}{ccc}
(i_0+m-1+(l-m)k)&+&(i_1+m-lk)\tau\\
m-1&+&m\tau
\end{array}\!\!\right)\cr
&\kern 2.1cm \text{(by \eqref{4.13} and the fact that $m<\tau$, which is easy to see)}\vrule height 0pt width 0pt depth 4mm \cr
=\,&
\begin{cases}
\displaystyle \binom{i_0+m-1+(l-m)k}{m-1}\binom{i_1+m-1-lk}m&\text{if}\ 0\le l\le m-1,\vspace{3mm}\cr
\displaystyle \binom{i_0+m-1+(l-m)k}{m-1}\binom{i_1+m-lk}m&\text{if}\ m\le l\le r-1,
\end{cases}\cr
&\kern 7.5cm \text{(by \eqref{4.14} and \eqref{4.15})}.
\end{split}
\end{equation}
Combining \eqref{4.12} and \eqref{4.16} gives
\[
\begin{split}
g_\alpha(x)=\,&(x-1)\sum_{\substack{0\le i_0\le m-1\cr 0\le i_1\le m}}(-1)^{i_0+i_1}\binom{m-1}{i_0}\binom m{i_1}\frac{x^{r(i_0+i_1)}-1}{x^r-1}\cr
&\cdot\Bigl[\sum_{l=0}^{m-1}\binom{i_0+m-1+(l-m)k}{m-1}\binom{i_1+m-1-lk}m x^l\cr
&+\sum_{l=m}^{r-1}\binom{i_0+m-1+(l-m)k}{m-1}\binom{i_1+m-lk}m x^l\Bigr]-1\cr
=\,&(x-1)\sum_{\substack{0\le i_0\le m-1\cr 0\le i_1\le m}}(-1)^{i_0+i_1}\binom{m-1}{i_0}\binom m{i_1}\frac{x^{r(i_0+i_1)}-1}{x^r-1}\cr
&\cdot\Bigl[\sum_{l=0}^{r-1}\binom{i_0+m-1+(l-m)k}{m-1}\binom{i_1+m-lk}m x^l\cr
&-\sum_{l=0}^{m-1}\binom{i_0+m-1+(l-m)k}{m-1}\binom{i_1+m-1-lk}{m-1} x^l\Bigr]-1.
\end{split}
\]
\end{proof}

\begin{lem}\label{L4.3} Let $x\in\overline\f_t$ be such that $x^r\in\f_\tau^*\setminus\{1\}$. Assume that $k:=(\tau+1)/r>2$. Then 
\begin{equation}\label{4.17}
\begin{split}
g_{1+\tau+\tau^3}(x)=\,&(x-1)\sum_{0\le i_0,i_1,i_3\le 1}(-1)^{i_0+i_1+i_3}\frac{x^{r(i_0+i_1+i_3)}-1}{x^r-1}\cr
&\cdot\Bigl[\sum_{l=0}^{r-1}\bigl(i_0+1+(l-2)k\bigr)\bigl(i_1+1-(l-1)k\bigr)(i_3+1-lk)x^l\cr
&-(i_0+1-2k)(i_1+i_3+k+1)-(i_0+1-k)(i_3+1-k)x\Bigr]-1.
\end{split}
\end{equation}
\end{lem}

\begin{proof}
Let $\alpha=1+\tau+\tau^3$. For any integer $i\ge 0$, we have $\binom\alpha i=0$ unless $i=i_0+i_1\tau+i_3\tau^3$, where $0\le i_0,i_1,i_3\le 1$, in which case,
\[
\binom\alpha i=1.
\]
Thus \eqref{2.12} gives
\begin{equation}\label{4.18}
\begin{split}
&g_\alpha(x)=\cr
&(x-1)\sum_{l=0}^{r-1}x^l\sum_{0\le i_0,i_1,i_3\le 1}(-1)^{i_0+i_1+i_3}\binom{i_0+i_1\tau+i_3\tau^3+\alpha-\frac{\alpha+1}r+\frac lr}\alpha\frac{x^{r(i_0+i_1+i_3)}-1}{x^r-1}-1.
\end{split}
\end{equation}
We have
\[
\begin{split}
&i_0+i_1\tau+i_3\tau^3+\alpha-\frac{\alpha+1}r+\frac lr\cr
\equiv\,&i_0+i_1\tau+i_3\tau^3+1+\tau+\tau^3-k(2-\tau+\tau^2)+lk(1-\tau+\tau^2-\tau^3)\pmod{\tau^4}\cr
=\,&\bigl(i_0+1+(l-2)k\bigr)+\bigl(i_1+1+k(l-1)(\tau-1)\bigr)\tau+(i_3+1-lk)\tau^3.
\end{split}
\]
When $0\le l\le 1$,
\[
i_0+1+(l-2)k+\tau\le 2-k+\tau<\tau,
\]
and
\[
i_0+1+(l-2)k+\tau\ge 1-2k+\tau=\tau+1-\frac 2r(\tau+1)\ge 0.
\]
So 
\begin{equation}\label{4.19}
\begin{split}
&\binom{i_0+i_1\tau+i_3\tau^3+\alpha-\frac{\alpha+1}r+\frac lr}\alpha\cr
=\,&\bigl(i_0+1+(l-2)k\bigr)\left(\!\!
\begin{array}{ccc}
\bigl(i_1+k(l-1)(\tau-1)\tau&+&(i_3+1-lk)\tau^3\cr
\tau&+&\tau^3
\end{array}\!\!\right).
\end{split}
\end{equation}
When $l=0$,
\[
\begin{cases}
i_1+k(l-1)(\tau-1)+\tau^2\ge-k(\tau-1)+\tau^2=\displaystyle -\frac 1r(\tau+1)(\tau-1)+\tau^2\ge 0,\vspace{1mm}\cr
i_1+k(l-1)(\tau-1)+\tau^2\le 1-k(\tau-1)+\tau^2<\tau^2.
\end{cases}
\]
When $l=1$,
\[
\begin{cases}
i_1+k(l-1)(\tau-1)\ge 0,\cr
i_1+k(l-1)(\tau-1)\le 1<\tau^2.
\end{cases}
\]
Hence
\begin{equation}\label{4.20}
\left(\!\!
\begin{array}{ccc}
\bigl(i_1+k(l-1)(\tau-1)\tau&+&(i_3+1-lk)\tau^3\cr
\tau&+&\tau^3
\end{array}\!\!\right)=
\begin{cases}
(i_1+k)i_3&\text{if}\ l=0,\vspace{2mm}\cr
i_1(i_3+1-k)&\text{if}\ l=1.
\end{cases}
\end{equation}
When $2\le l\le r-1$,
\[
\begin{cases}
i_0+1+(l-2)k\ge 0,\cr
i_0+1+(l-2)k\le 2+(r-3)k=2-3k+\tau+1<\tau,
\end{cases}
\]
and
\[
\begin{cases}
i_1+1+k(l-1)(\tau-1)&\kern-2mm\ge 0,\vspace{2mm}\cr
i_1+1+k(l-1)(\tau-1)&\kern-2mm\le 2+k(r-3)(\tau-1)\le 2-3+kr(\tau-1)\cr
&\kern-2mm<(\tau+1)(\tau-1)<\tau^2.
\end{cases}
\]
Hence
\begin{equation}\label{4.21}
\binom{i_0+i_1\tau+i_3\tau^3+\alpha-\frac{\alpha+1}r+\frac lr}\alpha=\bigl(i_0+1+(l-2)k\bigr)\bigl(i_1+1-(l-1)k\bigr)(i_3+1-lk).
\end{equation}
Gathering \eqref{4.19} -- \eqref{4.21}, we have 
\begin{equation}\label{4.22}
\begin{split}
&\binom{i_0+i_1\tau+i_3\tau^3+\alpha-\frac{\alpha+1}r+\frac lr}\alpha\cr
=\,&\bigl(i_0+1+(l-2)k\bigr)\cdot 
\begin{cases}
(i_1+k)i_3&\text{if}\ l=0,\cr
i_1(i_3+1-k)&\text{if}\ l=1,\cr
\bigl(i_1+1-(l-1)k\bigr)(i_3+1-lk)&\text{if}\ 2\le l\le r-1.
\end{cases}
\end{split}
\end{equation}
Now by \eqref{4.18} and \eqref{4.22},
\[
\begin{split}
g_\alpha(x)=\,&(x-1)\sum_{0\le i_0,i_1,i_3\le 1}(-1)^{i_0+i_1+i_3}\frac{x^{r(i_0+i_1+i_3)}-1}{x^r-1}\cr
&\cdot\Bigl[(i_0+1-2k)(i_1+k)i_3+(i_0+1-k)i_1(i_3+1-k)x\cr
&+\sum_{l=2}^{r-1}\bigl(i_0+1+(l-2)k\bigr)\bigl(i_1+1-(l-1)k\bigr)(i_3+1-lk)x^l\Bigr]-1\cr
=\,&(x-1)\sum_{0\le i_0,i_1,i_3\le 1}(-1)^{i_0+i_1+i_3}\frac{x^{r(i_0+i_1+i_3)}-1}{x^r-1}\cr
&\cdot\Bigl[\sum_{l=0}^{r-1}\bigl(i_0+1+(l-2)k\bigr)\bigl(i_1+1-(l-1)k\bigr)(i_3+1-lk)x^l\cr
&-(i_0+1-2k)(i_1+i_3+k+1)-(i_0+1-k)(i_3+1-k)x\Bigr]-1.
\end{split}
\]
\end{proof}

Recall from \eqref{2.15} that for positive integers $\alpha$ with $\alpha\equiv-1\pmod r$ we have $h_\alpha=\x^{-1}g_\alpha\in(\Bbb Z[1/r])[\x]$.

\begin{lem}\label{L4.4}
Let $r\ge 3$. Assume that $\tau$ is a prime such that $\tau\equiv-1\pmod r$, $k:=(\tau+1)/r>2$ and $\tau\nmid\rho(r)$, where $\rho(r)$ is a nonzero integer, depending only on $r$, to be defined in \eqref{4.30}. Then
\begin{equation}\label{4.22.1}
\text{\rm gcd}_{\f_\tau[\x]}(h_\tau,h_{\tau^3},h_{\tau^4-2},h_{1+\tau+\tau^3},h_{2\tau+1})=1.
\end{equation}
\end{lem} 

\begin{proof}
$1^\circ$ We claim that $h_\tau(0)\ne 0$ in $\f_\tau$.

Using the first step of \eqref{4.3}, we have
\[
-g'_\tau(0)=1+(-1)(2-k)=-1+k.
\]
Since $k=(\tau+1)/r$, we have $1<k<\tau$. Thus $g'_\tau(0)\ne 0$ and hence $h_\tau(0)\ne 0$.

\medskip
$2^\circ$ Assume to the contrary that $h_\tau,h_{\tau^3},h_{\tau^4-2},h_{1+\tau+\tau^3},h_{2\tau+1}$ have a common root $x\in\overline\f_\tau$. By $1^\circ$, $x\ne 0$. By Lemma~\ref{4.1}, $x^r\in\f_\tau^*\setminus\{1\}$, $(r+1)x-r\ne 0$, and \eqref{4.1} holds.

By Lemma~\ref{L4.3}, we have \eqref{4.17}. We need to rewrite \eqref{4.17} in a form suitable for further computation. Let $y=x^r$. For any integer $k\ge 0$, define
\begin{equation}\label{4.23}
s_k=\sum_{l=0}^{r-1}l^k\x^l\in\Bbb Z[\x].
\end{equation}
(Note that $s_k$ also depends on $r$. Since $r$ is fixed, it is suppressed in the notation.) The polynomial $s_k$ satisfies the recursive relation
\begin{equation}\label{4.24}
\begin{cases}
s_0=\displaystyle\frac{\x^r-1}{\x-1},\cr
s_k=\displaystyle\frac 1{\x-1}\Bigl[(-1)^k\sum_{j=0}^{k-1}(-1)^j\binom kj s_j-(-1)^k+(r-1)^k\x^r\Bigr],\quad k>0.
\end{cases}
\end{equation}
We can rewrite \eqref{4.17} as
\begin{equation}\label{4.25}
\begin{split}
&g_{1+\tau+\tau^3}(x)\cr
=\,&(x-1)\sum_{0\le i_0,i_1,i_3\le 1}(-1)^{i_0+i_1+i_3}\frac{y^{i_0+i_1+i_3}-1}{y-1}\cr
&\cdot \bigl[a_0s_0(x)+a_1s_1(x)+a_2s_2(x)+a_3s_3(x)\cr
&-(i_0+1-2k)(i_1+i_3+k+1)-(i_0+1-k)(i_3+1-k)x\bigr]-1,
\end{split}
\end{equation}
where
\begin{equation}\label{4.26}
\begin{cases}
a_0=\,&1-k-2k^2+i_0+ki_0+i_1-2ki_1+i_0i_1+i_3-ki_3-2k^2i_3\cr
&+i_0i_3+ki_0i_3+i_1i_3-2ki_1i_3+i_0i_1i_3,\cr
a_1=\,&-k+4k^2+2k^3-2ki_0-k^2i_0+2k^2i_1-ki_0i_1+3k^2i_3-ki_0i_3+ki_1i_3,\cr
a_2=\,&-k^2-3k^3+k^2i_0-k^2i_1-k^2i_3,\cr
a_3=\,&k^3.
\end{cases}
\end{equation}
Combining \eqref{4.25} and \eqref{4.26} allows us to express $g_{1+\tau+\tau^3}(x)$ as a rational function in $x$ and $y$, and, with the substitutions $k=1/r$ and 
\[
y=\frac{x(rx+1-r)}{(r+1)x-r},
\]
that expression becomes
\begin{equation}\label{4.27}
g_{1+\tau+\tau^3}(x)=\frac{x(x-1)^2}{\bigl((r+1)x-r\bigr)^3}G_{1+\tau+\tau^3}(x),
\end{equation}
where 
\begin{equation}\label{4.28}
\begin{split}
&G_{1+\tau+\tau^3}(\x)\cr
=\,&\left(4 r^3-3 r+1\right) \x^4+\left(-9 r^3+6 r^2+5 r-1\right)
   \x^3+\left(6 r^3-5 r^2-7 r-1\right) \x^2\cr
&+\left(-r^3-2 r^2+11 r-7\right)\x+r^2 -3 r+2
\end{split}
\end{equation}

In the same way, and starting from \eqref{4.1} with $m=2$, we find that
\begin{equation}\label{4.28.1}
2g_{1+2\tau}(x)=\frac{x(x-1)^3}{\bigl((r+1)x-r\bigr)^3}G_{1+2\tau}(x),
\end{equation}
where 
\begin{equation}\label{4.29}
\begin{split}
G_{1+2\tau}(\x)
=\,&\left(12 r^3+2 r^2-8 r+2\right) \x^3+\left(-28 r^3+31 r^2+16
   r-7\right) \x^2\cr
&+\left(20 r^3-44 r^2+14 r+6\right) \x -4 r^3+11 r^2 -6 r-1.
\end{split}
\end{equation}

The resultant of $G_{1+\tau+\tau^3}$ and $G_{1+2\tau}$ is 
\begin{equation}\label{4.30}
\begin{split}
\rho(r):=\,&R(G_{1+\tau+\tau^3},G_{1+2\tau})\cr
=\,&-(r-2)(r-1)^2(r+1)(2r-1)(2616 r^{10}-4994 r^9+212 r^8-21785 r^7\cr
&+58174 r^6-30241 r^5+2258 r^4-30963
   r^3+12562 r^2-433 r-78 ).
\end{split}
\end{equation}
Since $r\ge 3$, we have $\rho(r)\ne 0$ in $\Bbb Z$. (It is easy to show that $\rho(r)$, as a polynomial in $r$, has no integer roots $\ge 3$.) Therefore we may choose $\tau$ such that $\tau\nmid \rho(r)$. However, since $x$ is a common root of $g_{1+\tau+\tau^3}$ and $g_{1+2\tau}$, it is a common root of $G_{1+\tau+\tau^3}$ and $G_{1+2\tau}$, which contradicts the fact that $R(G_{1+\tau+\tau^3},G_{1+2\tau})\not\equiv 0\pmod\tau$.
\end{proof}

\begin{lem}\label{L4.5}
Assume $r\ge 3$. We have
\[
\text{\rm gcd}_{\Bbb Q[\x]}\{h_\alpha:\alpha>0,\ \alpha\equiv -1\pmod r\}=1.
\]
\end{lem}

\begin{proof} Assume the contrary; that is, there exists a primitive polynomial $d\in\Bbb Z[\x]$ with $\deg d>0$ such that $d\mid h_\alpha$ in $\Bbb Q[\x]$ for all $\alpha>0$ with $\alpha\equiv -1\pmod r$. Choose a prime $\tau$ satisfying the conditions in Lemma~\ref{L4.4} such that $\tau$ does not divide the leading coefficient of $d$, that is, $\deg_{\f_\tau[\x]}d=\deg_{\Bbb Q[\x]}d$. For all $\alpha>0$ with $\alpha\equiv -1\pmod r$, since $h_\alpha\in(\Bbb Z[1/r])[\x]$, we have $d\mid h_\alpha$ in $\f_\tau[\x]$. This a contradiction to \eqref{4.22.1}.
\end{proof}

\begin{lem}\label{L4.6}
Fix $r\ge 3$, and assume that $f$ is a PP of $\f_{q^2}$, where $q\ge r$, $q+1\equiv 0\pmod r$ and $a^{q+1}\ne 1$ ($a\in\f_{q^2}^*$). Then there are only finitely many possibilities for $p:=\text{\rm char}\,\f_q$.
\end{lem}

\begin{proof}
By Lemma~\ref{L4.5}, there exist $0<\alpha_1<\cdots<\alpha_m$ with $\alpha_i\equiv -1\pmod r$, $1\le r\le m$, such that
\[
\text{gcd}_{\Bbb Q[\x]}(h_{\alpha_1},\dots,h_{\alpha_m})=1.
\]
Therefore, there exist $a_1,\dots,a_m\in\Bbb Z[\x]$ such that 
\begin{equation}\label{4.31}
a_1 h_{\alpha_1}+\cdots+a_m h_{\alpha_m}=A\in\Bbb Z\setminus\{0\}.
\end{equation}
If $q<(r-1)(\alpha_m+2)$, there are only finitely many possibilities for $p$. If $q\ge (r-1)(\alpha_m+2)$, by Lemma~\ref{L2.4}, $h_{\alpha_1},\dots,h_{\alpha_m}$ have a common root in $\overline\f_p$. It follows from \eqref{4.31} that $p\mid A$. Hence there are only finitely many possibilities for $p$.
\end{proof}


\subsection{Finiteness of the possibilities of $q$ in a given characteristic}\

\begin{lem}\label{L4.7}
Fix $r\ge 3$, and let $p$ be a prime such that there is a power $\tau$ of $p$ with $\tau\equiv -1\pmod r$. If $f$ is a PP of $\f_{q^2}$, where $\text{char}\,\f_q=p$, $q\equiv -1\pmod r$, $q\ge \tau^4$ and $a^{q+1}\ne 1$ ($a\in\f_{q^2}^*$), then $p\in\{2,3,5\}$.
\end{lem} 

\begin{proof}
Since $f$ is a PP of $\f_{q^2}$ and $q\ge \tau^4$, it follows that $g_\alpha$, $\alpha\in\{\tau, 1+2\tau, 2+3\tau, 3+4\tau, 1+\tau+\tau^3, -2+\tau^4\}$, have a common root $x\in\overline\f_p\setminus\{0\}$  with $x^r\in\f_\tau^*\setminus\{1\}$. By Lemma~\ref{L4.1}, $(r+1)x-r\ne 0$, and \eqref{4.1} holds. The equations \eqref{4.27} and \eqref{4.28.1} remain valid. Moreover, $g_{2+3\tau}(x)$ and $g_{3+4\tau}(x)$ can be computed in a way similar to the computation of $g_{1+2\tau}(x)$. To sum up, we have 
\begin{equation}\label{4.32}
g_{1+\tau+\tau^3}(x)=\frac{x(x-1)^2}{\bigl((r+1)x-r\bigr)^3}G_{1+\tau+\tau^3}(x),
\end{equation}
\begin{equation}\label{4.33}
2g_{1+2\tau}(x)=\frac{x(x-1)^3}{\bigl((r+1)x-r\bigr)^3}G_{1+2\tau}(x),
\end{equation}
\begin{equation}\label{4.34}
12g_{2+3\tau}(x)=\frac{(x-1)^2}{\bigl((r+1)x-r\bigr)^5}G_{2+3\tau}(x),
\end{equation}
\begin{equation}\label{4.35}
144g_{3+4\tau}(x)=\frac{(x-1)^2}{\bigl((r+1)x-r\bigr)^7}G_{3+4\tau}(x),
\end{equation}
where
\begin{equation}\label{4.36}
\begin{split}
&G_{1+\tau+\tau^3}(\x)\cr
=\,&\left(4 r^3-3 r+1\right) \x^4+\left(-9 r^3+6 r^2+5 r-1\right)
   \x^3+\left(6 r^3-5 r^2-7 r-1\right) \x^2\cr
&+\left(-r^3-2 r^2+11 r-7\right)\x+r^2 -3 r+2,
\end{split}
\end{equation}
\begin{equation}\label{4.37}
\begin{split}
G_{1+2\tau}(\x)
=\,&\left(12 r^3+2 r^2-8 r+2\right) \x^3+\left(-28 r^3+31 r^2+16
   r-7\right) \x^2\cr
&+\left(20 r^3-44 r^2+14 r+6\right) \x -4 r^3+11 r^2 -6 r-1,
\end{split}
\end{equation}
\begin{equation}\label{4.38}
\begin{split}
&G_{2+3\tau}(\x)\cr
=\,&2 \left(480 r^5-56 r^4-612 r^3+574 r^2-210 r+4\right) \x^8\cr
&+2 \left(-2904
   r^5+2390 r^4+2655 r^3-4033 r^2+2004 r-238\right) \x^7\cr
&+2 \left(7464
   r^5-11148 r^4-1230 r^3+9711 r^2-6477 r+1266\right) \x^6\cr
&+2 \left(-10572
   r^5+22426 r^4-9387 r^3-8690 r^2+9177 r-2414\right) \x^5\cr
&+2 \left(8940
   r^5-24118 r^4+19116 r^3-502 r^2-5736 r+1958\right) \x^4\cr
&+2 \left(-4560
   r^5+14658 r^4-15783 r^3+5571 r^2+1002 r-762\right) \x^3\cr
&+2 \left(1344
   r^5-4928 r^4+6354 r^3-3305 r^2+375 r+142\right) \x^2\cr
&+2 \left(-204
   r^5+830 r^4-1197 r^3+728 r^2-147 r-10\right) \x\cr
&+2 \left(12 r^5-54 r^4+84
   r^3-54 r^2+12 r\right),
\end{split}
\end{equation}
\begin{equation}\label{4.39}
\begin{split}
&G_{3+4\tau}(\x)\cr
=\,&\bigl(226800
   r^7-211140 r^6-188340 r^5+388485 r^4-255456 r^3+74058 r^2-6204r\cr
&+597\bigr) \x^{12}
+\bigl(-2030400 r^7+3135420 r^6+355372 r^5-3641401
   r^4+3175096 r^3\cr
&-1217746 r^2+175660 r-7873\bigr) \x^{11}
+\bigl(8169120
   r^7-17438796 r^6+6486732 r^5\cr
&+12248029 r^4-15562968 r^3+7606042
   r^2-1549092 r+94501\bigr) \x^{10}\cr
&+\bigl(-19474560 r^7+52594740
   r^6-39521644 r^5-13998585 r^4+38790944 r^3\cr
&-24354306 r^2+6413804
   r-556953\bigr) \x^9
+\bigl(30522960 r^7-98911656 r^6\cr
&+106243584
   r^5-18404742 r^4-50668800 r^3+44318196 r^2-14546448 r\cr
&+1649658\bigr)
   \x^8
+\bigl(-32977152 r^7+123769656 r^6-168070536 r^5+81291294
   r^4\cr
&+25818720 r^3-46698708 r^2+19296552 r-2678946\bigr)
   \x^7
+\bigl(25026624 r^7\cr
&-106072632 r^6+170852376 r^5-120540462
   r^4+17400096 r^3+26245236 r^2\cr
&-15232440 r+2516178\bigr)
   \x^6
+\bigl(-13329792 r^7+62623944 r^6-115027224 r^5\cr
&+100925334
   r^4-36600336 r^3-4219764 r^2+6945288 r-1415658\bigr) \x^5\cr
&+\bigl(4888080
   r^7-25116660 r^6+51237324 r^5-52014495 r^4+25779696 r^3\cr
&-3616638
   r^2-1606668 r+481329\bigr) \x^4
+\bigl(-1183680 r^7+6598476 r^6\cr
&-14687844 r^5
+16566267 r^4-9691800 r^3+2414310 r^2+74028 r-96093\bigr)
   \x^3\cr
&+\bigl(175392 r^7-1059516 r^6+2548316 r^5-3117839 r^4+2021816
   r^3-621326 r^2\cr
&+43484 r+10249\bigr) \x^2
+\bigl(-13824 r^7+91332
   r^6-236988 r^5+310403 r^4\cr
&-215952 r^3+73910 r^2-8436 r-445\bigr) \x
+432 r^7-3168 r^6+8872 r^5-12288 r^4\cr
&+8944 r^3-3264 r^2+472r.
\end{split}
\end{equation}

Assume to the contrary that $p\notin\{2,3,5\}$.

\medskip
$1^\circ$ We claim that $r\ne 1,2$ in $\f_p$.

If $r=1$ in $\f_p$, by \eqref{4.36} and \eqref{4.37},
\[
\begin{split}
G_{1+\tau+\tau^3}(\x)\,&=\x(1-7\x+\x^2+2\x^3),\cr
G_{1+2\tau}(\x)\,&=2(1-2\x+6\x^2+4\x^3).
\end{split}
\]
We have
\[
A\cdot(1-7\x+\x^2+2\x^3)+B\cdot(1-2\x+6\x^2+4\x^3)=3\cdot 31,
\]
where 
\[
A=4(-5+26\x+17\x^2),\qquad B=113-18\x-34\x^2.
\]
When $p\ne 31$, this contradicts the fact that $x$ is a common root of $G_{1+\tau+\tau^3}$ and $G_{1+2\tau}$. When $p=31$, we find that 
\[
\text{gcd}_{\f_{31}[\x]}\bigl(1-7\x+\x^2+2\x^3,\; G_{2+3\tau}(\x)\bigr)=1,
\]
which is also a contradiction. 

Similarly, if $r=2$ in $\f_p$, we have
\[
\begin{split}
G_{1+\tau+\tau^3}(\x)\,&=\x(\x-1)(3\x-1)(9\x-1),\cr
G_{1+2\tau}(\x)\,&=1+18\x-75\x^2+90\x^3,
\end{split}
\]
and
\[
A\cdot(3\x-1)(9\x-1)+B\cdot(1+18\x-75\x^2+90\x^3)=2\cdot 89,
\]
where
\[
A=101-210\x-120\x^2,\qquad B=77+36\x.
\]
When $p\ne 89$, we have a contradiction. When $p=89$, we find that 
\[
\text{gcd}_{\f_{89}[\x]}\bigl((3\x-1)(9\x-1),\; G_{2+3\tau}(\x)\bigr)=1,
\]
which is again a contradiction.

\medskip
$2^\circ$ We computed the resultant $R(G_\alpha, G_\beta)$ for all $\alpha,\beta\in\{1+\tau+\tau^3,\, 1+2\tau,\, 2+3\tau,\, 3+4\tau\}$, $\alpha\ne \beta$. The following selected results are to used in the next step.
\[
\begin{split}
&R(G_{1+\tau+\tau^3}, G_{1+2\tau})\cr
=\,&-(r-2)(r-1)^2(r+1)(2r-1)(2616 r^{10}-4994 r^9+212 r^8-21785 r^7\cr
&+58174 r^6-30241 r^5+2258 r^4-30963
   r^3+12562 r^2-433 r-78 ),
\end{split}
\]
\[
\begin{split}
&R(G_{1+\tau+\tau^3}, G_{2+3\tau})\cr
=\,&-432(r-2)(r-1)^2 \bigl( 904200192 r^{24}-26766761472 r^{23}+331824091392 r^{22}\cr
&-2431984416768
   r^{21}+12259306904112 r^{20}-46191235200768 r^{19}\cr
&+136361704467312
   r^{18}-323448347180178 r^{17}+622072535670810 r^{16}\cr
&-963022456113483
   r^{15}+1163874205192488 r^{14}-1011383557230121 r^{13}\cr
&+458542507975098
   r^{12}+246769336119765 r^{11}-714942562087848 r^{10}\cr
&+750787901224431
   r^9-493424131713792 r^8+213252937500746 r^7\cr
&-55526043586236
   r^6+4871506549056 r^5+1905998384136 r^4\cr
&-727704486624 r^3+99138790224
   r^2-4740664936 r-30723984  \bigr),
\end{split}
\]
\[
\begin{split}
&R(G_{1+2\tau}, G_{2+3\tau})\cr
=\,&64(r-1)^2\bigl( 3034644480 r^{24}-32529595392 r^{23}+657287337600 r^{22}\cr
&-8299978411200
   r^{21}+56251922766208 r^{20}-235334307782820 r^{19}\cr
&+667348751791160
   r^{18}-1357351029780397 r^{17}+2048898884042996 r^{16}\cr
&-2339185451094502
   r^{15}+2030190962218068 r^{14}-1317041479048822 r^{13}\cr
&+596228225529184
   r^{12}-146983940173022 r^{11}-9832362841412 r^{10}\cr
&+18243146047288
   r^9-3480285918736 r^8-419376654818 r^7+134266774188 r^6\cr
&+4006393734
   r^5-501750288 r^4+168774858 r^3+13728780 r^2-279387 r-26676  \bigr),
\end{split}
\]
\begin{align*}
&R(G_{2+3\tau}, G_{3+4\tau})\cr
=\,&-2166612408926208(r-1)^4r\cr
&\bigl(  
563487703892102831444078594988441600
   r^{61}\cr
&-726125546232691779026874827707546337280
   r^{60}\cr
&+59045110242590913671812425635932636446720
   r^{59}\cr
&-1543744791046933080469394307821363420725248
   r^{58}\cr
&+17635449249902967455768645935833784434819072
   r^{57}\cr
&-13242432899897278662960866425758484852113408 
   r^{56}\cr
&-2713152307304734984112784294298843188484374528
   r^{55}\cr
&+49167113785520627862051507889328415696639492096
   r^{54}\cr
&-532005508889912404066224212737203577299136413696 
   r^{53}\cr &+4232941536010391592610119145834769263557334794240
   r^{52}\cr &-26668308257430003986301692253261472660483288334336
   r^{51}\cr &+138223679803630839988472126930234721490261656731648
   r^{50}\cr &-603371606772948308241133678928919829772581944918016
   r^{49}\cr &+2253263036525361481560649801335897811225932905111552
   r^{48}\cr &-7277627939189346602743525258539122393315807057145856   
   r^{47}\cr &+20482050195727748294625975698698552545832447378681856
   r^{46}\cr &-50469241326694886540303378829939924808940534437609472
   r^{45}\cr &+109116015700575168000559138444874729250457665549625344
   r^{44}\cr &-206833001455966937200391885170167363806880207397885440
   r^{43}\cr &+342067904435488966844808494037156843091882241623969024
   r^{42}\cr &-487811231347075610016324356763534517180367913784825216
   r^{41}\cr &+584233686801679434205060425125458216643529330954416288
   r^{40}\cr &-549554071172728542687909751281642717953409510206490544
   r^{39}\cr &+314223131189797714183700975565179370195444620345276672
   r^{38}\cr &+133581062853313781625442620551320263688679792640788912
   r^{37}\cr &-719005576269866309174336232009072883958528439583259936
   r^{36}\cr &+1294312372830884925102390208243282179982605190318985808
   r^{35}\cr &-1697083029482069169772622250712023278249847922213843920
   r^{34}\cr &+1822711929635067366866037068545821855495708719534629372
   r^{33}\cr &-1668072590352754776555877663405787031320991981331262266
   r^{32}\cr &+1321725918348941131951214783980893869255330910058016021
   r^{31}\cr &-911778936425557232224283842471797963685627465472456110
   r^{30}\cr &+546997770484831420252392693291747938643889360835758007
   r^{29}\cr &-283354905430826711294868071729328693370531366772786260
   r^{28}\cr &+124834211605604793196403392309875462890600770988324881
   r^{27}\cr &-45393447956785044057507553121051427403244796750690426
   r^{26}\cr &+12739990012931826396381326370346876043294784653218803
   r^{25}\cr &-2216341781295328849737524626206897519710481003439832
   r^{24}\cr &-111880124566904056707989666685359140477899452943199
   r^{23}\cr &+251522659926490301938210279823844297566176828690850
   r^{22}\cr &-101220601216222132408130216769135257528871064892053
   r^{21}\cr &+22624675343227091700489497078929605844367556828068
   r^{20}\cr &-1908972906824956271173431273409855682834097062467
   r^{19}\cr &-590736968567860662346653525293252204666240324570
   r^{18}\cr &+254845352675862913643564775851609223374663877663
   r^{17}\cr &-40310192660823848996399573166944148181773222708
   r^{16}\cr &+424814302621371238653289731946939032578004175
   r^{15}\cr &+898345480967916525141130472885228750811904742
   r^{14}\cr &-94674423996970582089569701194783717008390539
   r^{13}\cr &-14511771149946680482629782799033859030393612
   r^{12}\cr &+2734787899038559469591406410941339153490723
   r^{11}\cr &+423933198428320883122423214907073858661954
   r^{10}\cr &-142158801188850304531189733212379758358615 
   r^9\cr &+8803657260993483198112895867761687002168
   r^8\cr & +708140687476573312595409552324488896491
   r^7\cr & -57406290119794767351844048706054692826
   r^6\cr & -2481072222126354679450524375955264519
   r^5\cr & -416547665183611789472998210835162020
   r^4\cr & +39478358602778799084783256284417023
   r^3\cr & -165550574731176119297479940971294
   r^2\cr & +116564567854207875652415637207689 r\cr & -6073037054382858461471337919250
\bigr).
\end{align*}
Let $R'(G_\alpha,G_\beta)$ denote the expression obtained from $R(G_\alpha,G_\beta)$ with the factors $r$, $r-1$ and $r-2$ removed. For $\alpha_1,\beta_1,\alpha_2,\beta_2\in\{1+\tau+\tau^3,\, 1+2\tau,\, 2+3\tau,\, 3+4\tau\}$, we treat $R'(G_{\alpha_1},G_{\beta_1})$ and $R'(G_{\alpha_2},G_{\beta_2})$ as polynomials in $r$ with integer coefficients, whose resultant $R(R'(G_{\alpha_1},G_{\beta_1}), R'(G_{\alpha_2},G_{\beta_2}))$ is thus an integer. We computed 
\[
\begin{split}
&R\bigl(R'(G_{1+\tau+\tau^3},G_{1+2\tau}), R'(G_{1+\tau+\tau^3},G_{2+3\tau})\bigr),\cr
&R\bigl(R'(G_{1+\tau+\tau^3},G_{1+2\tau}), R'(G_{1+2\tau},G_{2+3\tau})\bigr),\cr
&R\bigl(R'(G_{1+\tau+\tau^3},G_{1+2\tau}), R'(G_{2+3\tau},G_{3+4\tau})\bigr),
\end{split}
\]
each of which is a very large integer. However, the gcd of the above three integers is $2\cdot 3^2\cdot 5$, which is nonzero in $\f_p$. Thus we have a contradiction.
\end{proof}

\noindent{\bf Remark.}
In fact, we computed 
\begin{equation}\label{4.40}
R(R'(G_{\alpha_1},G_{\beta_1}), R'(G_{\alpha_2},G_{\beta_2}))
\end{equation}
for all $\alpha_1,\beta_1,\alpha_2,\beta_2\in \{1+\tau+\tau^3,\, 1+2\tau,\, 2+3\tau,\, 3+4\tau\}$ with $\alpha_1\ne\beta_1$, $\alpha_2\ne\beta_2$ and $\{\alpha_1,\beta_1\}\ne\{\alpha_2,\beta_2\}$. It turned out that the gcd of all such integers is $2\cdot 3^2\cdot 5$. Therefore, consideration of all the resultants in \eqref{4.40} still cannot eliminate the possibilities of $p=2,3,5$.

\medskip
The following corollary is immediate from Lemma~\ref{L4.7}

\begin{cor}\label{C4.8}
Fix $r\ge 3$ and a prime $p$ with $p\notin\{2,3,5\}$. Then there are only finitely many $q$ with $\text{\rm char}\,\f_q=p$ and $q\equiv-1\pmod r$ for which $f$ is a PP of $\f_{q^2}$, where $a^{q+1}\ne 1$ ($a\in\f_{q^2}^*$).
\end{cor}


\subsection{The cases $p=2,3,5$}\

\begin{lem}\label{L4.9} 
Fix $r\ge 3$, and let $p\in\{2,3,5\}$. Fix a power $\tau$ ($>2$) of $p$ such that $\tau\equiv -1\pmod r$, assuming that such a power exists. Let $q$ be a power of $p$ with $q\equiv -1\pmod r$. If $q\ge r\tau^4$, then $f$ is not a PP of $\f_{q^2}$, where $a^{q+1}\ne1$ ($a\in\f_{q^2}^*$).
\end{lem}

\begin{proof}
Assume to the contrary that $f$ is a PP of $\f_{q^2}$. It follows from Lemmas~\ref{L2.4} and \ref{L4.1} that $g_\alpha$, $\alpha\in\{\tau,\, 1+2\tau,\, 2+3\tau,\, 3+4\tau,\, 1+\tau+\tau^3,\, -2+\tau^4\}$ have a common root $x\in\overline\f_p\setminus\{0\}$ with $x^r\in\f_\tau^*\setminus\{1\}$ and $(r+1)x-r\ne 0$ that satisfies \eqref{4.1}.

\medskip
{\bf Case 1.} Assume $p=2$. Clearly, $r=1$ in $\f_2$. It follows that 
\[
G_{1+\tau+\tau^3}(\x)=\x(1+\x+\x^2).
\]
Since $p=2$, $G_{1+2\tau}$, $G_{2+3\tau}$ and $G_{3+4\tau}$ in \eqref{4.36} -- \eqref{4.38} are all $0$, due to the fact that they arise from $2g_{1+2\tau}(x)$, $12g_{2+3\tau}(x)$ and $144g_{3+4\tau}(x)$, respectively. It is possible to compute new polynomials arising in a similar way from $g_{1+2\tau}(x)$, $g_{2+3\tau}(x)$ and $g_{3+4\tau}(x)$. The result from $g_{1+2\tau}(x)$ is still $0$. The results from $g_{2+3\tau}(x)$, whose computation is detailed below, will be useful.

Let $k=(\tau+1)/r$ and $y=x^r$. By \eqref{4.11} (with $m=3$),
\begin{equation}\label{4.41}
\begin{split}
g_{2+3\tau}(x)=\,&(x-1)\sum_{\substack{0\le i_0\le 2\cr 0\le i_1\le 3}}(-1)^{i_0+i_1}\binom 2{i_0}\binom 3{i_1}\frac{y^{i_0+i_1}-1}{y-1}\cr
&\cdot\Bigl[\sum_{l=0}^{r-1}\binom{i_0+2+(l-3)k}2\binom{i_1+2-lk}3x^l\cr
&-\sum_{l=0}^2\binom{i_0+2+(l-3)k}2\binom{i_1+2-lk}2x^l\Bigr]-1.
\end{split}
\end{equation}

Let $F$ be a finite degree unramified extension of $\Bbb Q_2$ such that the finite field $\mathcal O_F/2\mathcal O_F$ contains $x$, where $\mathcal O_F$ is the ring of integers of $F$. Let $\frak x\in\mathcal O_F$ be a lift of $x$, and set $\frak y=\frak x^r$. Then 
\begin{equation}\label{4.42}
\frak y=\frac{\frak x(r\frak x+1-r)}{(r+1)\frak x-r}\pmod 2.
\end{equation}
Let $\frak g\in\mathcal O_F$ denote the expression obtained from \eqref{4.41} with $x$ and $y$ replaced by $\frak x$ and $\frak y$, respectively. Then $\frak g$ can be computed in a manner similar to \eqref{4.25}. We have
\begin{equation}\label{4.43}
\begin{split}
\frak g=\,&(\frak x-1)\sum_{\substack{0\le i_0\le 2\cr 0\le i_1\le 3}}(-1)^{i_0+i_1}\binom 2{i_0}\binom 3{i_1}\frac{\frak y^{i_0+i_1}-1}{\frak y-1}\cr
&\cdot\Bigl[a_0s_0(\frak x)+a_1s_1(\frak x)+a_2s_2(\frak x)+a_3s_3(\frak x)+a_4s_4(\frak x)+a_5s_5(\frak x)\cr
&-\sum_{l=0}^2\binom{i_0+2+(l-3)k}2\binom{i_1+2-lk}2\frak x^l\Bigr]-1,
\end{split}
\end{equation} 
where $s_i$ is given by \eqref{4.23} and \eqref{4.24}, and
\begin{equation}\label{4.44}
\begin{split}
a_0=\,& \frac{i_0^2 i_1^3}{12}+\frac{i_0^2 i_1^2}{2}+\frac{11 i_0^2
   i_1}{12}+\frac{i_0^2}{2}-\frac{1}{2} i_0 i_1^3 k+\frac{i_0 i_1^3}{4}-3 i_0 i_1^2
   k+\frac{3 i_0 i_1^2}{2}-\frac{11 i_0 i_1 k}{2}\cr
&+\frac{11 i_0 i_1}{4}-3 i_0 k+\frac{3
   i_0}{2}+\frac{3 i_1^3 k^2}{4}-\frac{3 i_1^3 k}{4}+\frac{i_1^3}{6}+\frac{9 i_1^2
   k^2}{2}-\frac{9 i_1^2 k}{2}+i_1^2\cr
&+\frac{33 i_1 k^2}{4}-\frac{33 i_1
   k}{4}+\frac{11 i_1}{6}+\frac{9 k^2}{2}-\frac{9 k}{2}+1, \cr
a_1= \,& -\frac{1}{4} i_0^2 i_1^2 k-i_0^2 i_1 k-\frac{11 i_0^2 k}{12}+\frac{1}{6} i_0 i_1^3
   k+\frac{3}{2} i_0 i_1^2 k^2+\frac{1}{4} i_0 i_1^2 k+6 i_0 i_1 k^2-\frac{7 i_0 i_1
   k}{6}\cr
&+\frac{11 i_0 k^2}{2}-\frac{7 i_0 k}{4}-\frac{1}{2} i_1^3 k^2+\frac{i_1^3
   k}{4}-\frac{9 i_1^2 k^3}{4}-\frac{3 i_1^2 k^2}{4}+i_1^2 k-9 i_1 k^3+\frac{7 i_1
   k^2}{2}\cr
&+\frac{3 i_1 k}{4}-\frac{33 k^3}{4}+\frac{21 k^2}{4}-\frac{k}{3}, \cr
a_2= \,& \frac{1}{4} i_0^2 i_1 k^2+\frac{i_0^2 k^2}{2}-\frac{1}{2} i_0 i_1^2 k^2-\frac{3}{2}
   i_0 i_1 k^3-\frac{5}{4} i_0 i_1 k^2-3 i_0 k^3-\frac{i_0 k^2}{3}+\frac{i_1^3
   k^2}{12}\cr
&+\frac{3 i_1^2 k^3}{2}-\frac{i_1^2 k^2}{4}+\frac{9 i_1
   k^4}{4}+\frac{15 i_1 k^3}{4}-\frac{19 i_1 k^2}{12}+\frac{9
   k^4}{2}+k^3-\frac{5 k^2}{4}, \cr
a_3= \,& -\frac{i_0^2 k^3}{12}+\frac{1}{2} i_0 i_1 k^3+\frac{i_0 k^4}{2}+\frac{3 i_0
   k^3}{4}-\frac{i_1^2 k^3}{4}-\frac{3 i_1 k^4}{2}-\frac{i_1 k^3}{4}-\frac{3
   k^5}{4}-\frac{9 k^4}{4}+\frac{5 k^3}{12}, \cr
a_4= \,& -\frac{i_0 k^4}{6}+\frac{i_1 k^4}{4}+\frac{k^5}{2}+\frac{k^4}{4}, \cr
a_5= \,& -\frac{k^5}{12}.
\end{split}
\end{equation}
Applying \eqref{4.24} and \eqref{4.44} to \eqref{4.43} gives
\begin{equation}\label{4.45}
\frak g=\frac 1{12(\frak x-1)^5}H(k,r,\frak x,\frak y),
\end{equation}
where $H(k,r,\frak x,\frak y)$ is a polynomial in $k,r,\frak x,\frak y$ with integer coefficients. (The expression $H(k,r,\frak x,\frak y)$ is too lengthy to be included here but can be easily generated with computer assistance.) Since $\tau>2$, we have $2^2\mid \tau$, and hence $k=(\tau+1)/r\equiv 1/r\pmod{2^2}$. So, the binomial coefficients in \eqref{4.41} remain the same modulo $2$ when $k$ is replaced by $1/r$. Therefore, 
\begin{equation}\label{4.46}
\begin{split}
\frak g\,&\equiv\frac 1{12(\frak x-1)^5}H(1/r,r,\frak x,\frak y)  \pmod 2\cr
&=\frac 1{r^5(\frak x-1)^5}H_1(r,\frak x,\frak y),
\end{split}
\end{equation}
where $H_1(r,\frak x,\frak y)=(r^5/12)H(1/r,r,\frak x,\frak y)$ is a polynomial in $\frak x$ and $\frak y$ whose coefficients are of the form $\frac 16C(r)$ with $C\in\Bbb Z[\x]$ and $\frac 16C(\Bbb Z)\subset\Bbb Z$. Making the substitution \eqref{4.42} in \eqref{4.46} gives
\begin{equation}\label{4.47}
\frak g\equiv\frac{(\frak x-1)^2}{\bigl((r+1)\frak x-r\bigr)^5}L(\frak x)\pmod 2,
\end{equation}
where
\begin{align*}
L(\x)=&
\frac16\Bigl[
\bigl(480 r^5-56 r^4-612 r^3+574 r^2-210
   r+4\bigr) \x^8\cr
&+\bigl(-2904 r^5+2390 r^4+2655 r^3-4033 r^2+2004
   r-238\bigr) \x^7\cr
&+\bigl(7464 r^5-11148 r^4-1230 r^3+9711 r^2-6477
   r+1266\bigr) \x^6\cr
&+\bigl(-10572 r^5+22426 r^4-9387 r^3-8690 r^2+9177
   r-2414\bigr) \x^5\cr
&+\bigl(8940 r^5-24118 r^4+19116 r^3-502 r^2-5736
   r+1958\bigr) \x^4\cr
&+\bigl(-4560 r^5+14658 r^4-15783 r^3+5571 r^2+1002
   r-762\bigr) \x^3\cr
&+\bigl(1344 r^5-4928 r^4+6354 r^3-3305 r^2+375
   r+142\bigr) \x^2\cr
&+\bigl(-204 r^5+830 r^4-1197 r^3+728 r^2-147
   r-10\bigr) \x \cr
&+12 r^5-54 r^4+84 r^3-54 r^2 +12 r
\Bigr]
\end{align*}

The coefficients of $L$ are also of the form $\frac 16C(r)$ with $C\in\Bbb Z[\x]$ and $\frac 16C(\Bbb Z)\subset\Bbb Z$. These coefficients, modulo $2$, only depend on $r$ modulo $4$. The reduction of \eqref{4.47} in $\mathcal O_F/2\mathcal O_F$ is
\begin{equation}\label{4.48}
g_{2+3\tau}(x)=\frac{(x-1)^2}{\bigl((r+1)x-r\bigr)^5}L(x).
\end{equation}  
(In fact, $L=\frac1{12}G_{2+3\tau}$ in $\Bbb Q[\x]$, where $G_{2+3\tau}$ is given in \eqref{4.38}. For $p=2$, \eqref{4.48} implies \eqref{4.34} but not the converse.)
We find that
\[
L(\x)=
\begin{cases}
\x^2(1+\x+\x^2+\x^4+\x^5)&\text{if}\ r\equiv 1\pmod 4,\cr
\x^4&\text{if}\ r\equiv -1\pmod 4,
\end{cases}
\]
and we always have
\[
\text{gcd}_{\f_2[\x]}(1+\x+\x^2,L)=1,
\]
This is a contradiction since $G_{1+\tau+\tau^3}(\x)=\x(1+\x+\x^2)$ and $L(\x)$ have a nonzero common root $x$.

\medskip
{\bf Case 2.} Assume $p=3$. First assume $r\equiv 1\pmod 3$. We have
\[
G_{1+\tau+\tau^3}(\x)=\x-\x^2+\x^3-\x^4=\x(1-\x)(\x^2+1).
\]
It follows that $x^2=-1$. Then $x^r=\pm 1$ or $\pm x$. By \eqref{4.1} we also have 
\[
x^r=\frac{x^2}{-x-1}=\frac 1{x+1}.
\]
If $x^r=\pm 1$, then $x+1=\pm 1$, and hence $x\in\f_3$. Thus $x^2\ne -1$, which is a contradiction. If $x^r=\pm x$, then $x+1=\mp x$, which forces $x=1$, which is also a contradiction.

Now assume $r\equiv-1\pmod 3$. We have $G_{1+\tau+\tau^3}(\x)=-\x+\x^2=\x(\x-1)$. This is impossible since $x\ne 0,1$.

\medskip
{\bf Case 3.} Assume $p=5$.

\medskip
{\bf Case 3.1.} Assume $r\equiv 1\pmod 5$. We have
\[
\begin{split}
G_{1+\tau+\tau^3}(\x)\,&=\x(1+3\x+\x^2+2\x^3),\cr
\frac 12G_{1+2\tau}(\x)\,&=\x(1+2\x+3\x^2),
\end{split}
\]
and 
\[
\text{gcd}_{\f_5[\x]}(1+3\x+\x^2+2\x^3,\; 1+2\x+3\x^2)=1.
\]
Hence we have a contradiction.

\medskip
{\bf Case 3.2.} Assume $r\equiv 2\pmod 5$. We have
\[
\begin{split}
\frac 12G_{1+2\tau}(\x)\,&=4+3\x,\cr
G_{2+3\tau}(\x)\,&=\x(4+4\x+3\x^2+\x^4+\x^5+3\x^7),
\end{split}
\]
and 
\[
\text{gcd}_{\f_5[\x]}(4+3\x,\; 4+4\x+3\x^2+\x^4+\x^5+3\x^7)=1.
\]
Hence we have a contradiction. (Note. In this case, $G_{1+\tau+\tau^3}(\x)=\x(1-\x)(4+3\x)$, which is a multiple of $G_{1+2\tau}(\x)$.)

\medskip
{\bf Case 3.3.} Assume $r\equiv -1\pmod 5$. We have $\frac 12G_{1+2\tau}(\x)=\x(\x-2)$. Thus $x=2$. However, by \eqref{4.1},
\[
x^r=\frac{x(rx+1-r)}{\bigl((r+1)x-r\bigr)}=0,
\]
which is a contradiction.

\medskip
{\bf Case 3.4.} Assume $r\equiv -2\pmod 5$. We have
\[
\begin{split}
G_{1+\tau+\tau^3}(\x)\,&=2+\x,\cr
\frac 12G_{1+2\tau}(\x)\,&=2+2\x+4\x^2,
\end{split}
\]
and 
\[
\text{gcd}_{\f_5[\x]}(2+\x,\; 2+2\x+4\x^2)=1.
\]
Hence we have a contradiction.
\end{proof}

In conclusion, Theorem~\ref{T1.5} follows from Lemma~\ref{L4.6}, Corollary~\ref{C4.8} and Lemma~\ref{L4.9}.


\end{document}